\documentclass[12pt,leqno]{article}
\usepackage[margin=1in]{geometry}
\usepackage{graphicx}
\usepackage{amsmath}
\usepackage{amsthm}
\usepackage{amssymb}
\usepackage{color}
\usepackage{ textcomp }
\usepackage{hyperref}
\usepackage{float}

\usepackage{algorithm}
\usepackage[noend]{algpseudocode}

\newcommand*\LH{\ensuremath{\overset{\kern2pt L'H}{=}}}

\newcommand{\Ba}{B_{\alpha}}
\newcommand{\Bb}{B_{\beta}}
\newcommand{\Bc}{B_{\gamma}}
\newcommand{\Wa}{W_{\alpha}}
\newcommand{\Wb}{W_{\beta}}
\newcommand{\Wc}{W_{\gamma}}

\newcommand{\Bbt}{\widetilde{B_{\beta}}}
\newcommand{\Bct}{\widetilde{B_{\gamma}}}
\newcommand{\Wat}{\widetilde{W_{\alpha}}}
\newcommand{\Wbt}{\widetilde{W_{\beta}}}
\newcommand{\Wct}{\widetilde{W_{\gamma}}}

\newcommand{\Ebt}{E_{\beta}}
\newcommand{\Ect}{E_{\gamma}}

\newcommand{\fp}[1]{\{#1\}}
\newcommand{\Fp}[1]{\left\{#1\right\}}
\newcommand{\fl}[1]{{\lfloor #1 \rfloor}}
\newcommand{\Fl}[1]{{\left\lfloor #1 \right\rfloor}}
\newcommand{\cl}[1]{{\lceil #1 \rceil}}
\newcommand{\Cl}[1]{{\left\lceil #1 \right\rceil}}

\newcommand{\ZZ}{\mathbb{Z}}
\newcommand{\NN}{\mathbb{N}}
\newcommand{\QQ}{\mathbb{Q}}
\newcommand{\RR}{\mathbb{R}}

\newcommand{\at}{\widetilde{a}}
\newcommand{\bt}{\widetilde{b}}
\newcommand{\ct}{\widetilde{c}}

\newcommand{\seq}[1]{(#1)_{n\in\NN}}
\newcommand{\Seq}[1]{\left(#1\right)_{n\in\NN}}

\newtheorem{lem}{Lemma}[section]
\newtheorem{prop}[lem]{Proposition}
\newtheorem{cor}[lem]{Corollary}

\newtheorem{thm}[lem]{Theorem}
\newtheorem{alg}[lem]{Algorithm}

\theoremstyle{definition}
\newtheorem{defn}[lem]{Definition}

\theoremstyle{remark}
\newtheorem{remark}[lem]{Remark}

\numberwithin{equation}{section}

\title{\bf Webster Sequences, Apportionment Problems, and Just-In-Time Sequencing}

\author{
XIAOMIN LI
}
\date{}

\newcommand{\Addresses}{{
  \bigskip
  \footnotesize
  Xiaomin Li, \textsc{Harvard University, Lu Group, Maxwell Dworkin, 33 Oxford St, Cambridge, MA 02138}\par\nopagebreak
  \textit{E-mail address}, Xiaomin Li:
  \href{mailto:xiaominli@g.harvard.edu}{\tt xiaominli@g.harvard.edu}

}}

\begin{document}

\maketitle

\begin{abstract}

    Given a real number $\alpha \in (0,1)$, we define the Webster sequence of density $\alpha$ to be $\Wa = \seq{\lceil(n-1/2)/\alpha\rceil}$, where $\cl{x}$ is the ceiling function. It is known that if $\alpha$ and $\beta$ are irrational with $\alpha + \beta = 1$, then $\Wa$ and $\Wb$ partition $\NN$. On the other hand, an analogous result for three-part partitions does not hold: There does not exist a partition of $\NN$ into sequences $\Wa, \Wb, \Wc$ with $\alpha, \beta, \gamma $ irrational.
    
    In this paper, we consider the question of how close one can come to a three-part partition of $\NN$ into Webster sequences with prescribed densities  $\alpha, \beta, \gamma $. We show that if  $\alpha, \beta, \gamma $ are irrational with  $\alpha + \beta +\gamma = 1$, there exists a partition of $\NN$ into sequences of densities $\alpha, \beta, \gamma$, in which one of the sequences is a Webster sequence and the other two are ``almost" Webster sequences that are obtained from Webster sequences by perturbing some elements by at most 1. We also provide two efficient algorithms to construct such partitions. The first algorithm outputs the $n$th element of each sequence in $O(1)$ time and the second algorithm gives the assignment of the $m$th positive integer to a sequence in $O(1)$ time. We show that the results are best-possible in several respects. Moreover, we describe applications of these results to apportionment and optimal scheduling problems. 

\end{abstract}

\section{Introduction}
\label{sec:introduction}

\subsection{Webster Sequences}
\label{subsec:WebsterSeq}
    
    A classical problem that arises in many contexts is the sequencing problem (see, for example, Kubiak \cite[p.~vii]{Kubiak2008}). Given a finite alphabet and a list of frequencies for these letters, we seek to construct a sequence over the given alphabet in which each letter occurs with its prescribed frequency as evenly as possible. This problem is equivalent to partitioning the set of natural numbers $\NN$ into sequences with prescribed densities as evenly as possible.
    
    For a single sequence with density $\alpha$, the ``ideal" position of the $n$th element is $n/\alpha$. However, $n/\alpha$ is in general not an integer, so the best one can do is to assign the $n$th element of the sequence to the nearest integer below or above $n/\alpha$, i.e., to  $\fl{n/\alpha}$ or $\cl{n/\alpha}$. This leads to the sequences
    \[
    \text{$B_\alpha = \Seq{\Fl{\frac{n}{\alpha}}}$ and $\overline{B}_\alpha =  \Seq{\Cl{\frac{n}{\alpha}}}$.}
    \] 
    In fact, the former sequence is known as the \emph{Beatty sequence of density $\alpha$} \cite{Beatty}. More generally, given a real number $\theta$, the \emph{non-homogeneous Beatty sequence of density $\alpha$ and phase $\theta$} \cite{OBryant2003} is defined as 
    \begin{equation*}
        B_{\alpha,\theta} = \left(\Fl{\frac{n+\theta}{\alpha}}\right)_{n\in\NN}.
    \end{equation*}
    Note that $B_\alpha = B_{\alpha, 0}$  and (assuming $\alpha\notin \QQ$) $\overline{B}_\alpha  = B_{\alpha,\alpha}$, so $B_\alpha$ and $\overline{B}_\alpha $ are both special cases of $  B_{\alpha,\theta}$.

    In this paper, we focus on another special case, namely the sequence (assuming $\alpha\notin \QQ$) 
      \begin{equation}
    \label{def-webster}
    \Wa = B_{\alpha, \alpha-1/2} =  \left(\Fl{ \frac{n+\alpha-1/2}{\alpha}}\right)_{n\in\NN} = \left(\Cl{ \frac{n-1/2}{\alpha}}\right)_{n\in\NN},
    \end{equation}
    which we call the \emph{Webster sequence of density $\alpha$}. The terminology is motivated by the connection with the Webster method of apportionment  (see Section \ref{subsec:ApportionMethods}). The Webster sequence represents, in a sense, the ``fairest" sequence among all sequences $ B_{\alpha,\theta} $ when measured by the quantity
     \begin{equation}
    \label{def:deviation-delta}
        \Delta_{\alpha, \theta}(m) = \#\{m'\le m: m'\in B_{\alpha,\theta}\}   - m \alpha.
    \end{equation}
    Indeed, it is easy to see (cf. Lemma \ref{lem:general-lemmas}(ii) below) that for the Webster sequence (i.e., when $\theta = \alpha - 1/2$), the quantity $\sup_m |\Delta_{\alpha, \theta}(m)|$ is minimal (and equal to $1/2$).
    
    As mentioned, the problem of sequencing letters from a given alphabet of size $k$ is equivalent to the problem of constructing a partition of $\NN$ into $k$ sequences with prescribed densities. In the case $k=2$, such partitions can be obtained using either Beatty or Webster sequences.

    \begin{thm}[Beatty's Theorem: Partition into two Beatty sequences \cite{Beatty}]
    \label{thm:beatty-k=2}
    Given positive irrational numbers $\alpha, \beta$ with $\alpha+\beta = 1$, the Beatty sequences $B_\alpha, B_\beta$ partition $\NN$.
    \end{thm}
    \begin{thm}[Partition into two Webster sequences \cite{Fraenkel1969, OBryant2003}]
    \label{thm:webster-k=2}
    Given positive irrational numbers $\alpha, \beta$ with $\alpha+\beta = 1$, the Webster sequences $\Wa, \Wb$ partition $\NN$.
    \end{thm}
    It is known that these theorems do not generalize to three sequences if their densities are irrational \cite{Uspensky1927, Skolem1957, Graham1963}. For rational densities $\alpha, \beta, \gamma$, a partition into non-homogeneous Beatty sequences exists only if $\{\alpha, \beta, \gamma\} = \{1/7, 2/7, 4/7\}$ (see \cite{Fraenkel1973} and \cite{Morikawa1982}). This result is a special case of Fraenkel's Conjecture (see \cite{Fraenkel1973} or \cite[Conjecture 6.22]{Kubiak2008}), which characterizes the $k$-tuples of densities $(\alpha_1, \dots, \alpha_k)$ for which a partition into $k$ non-homogeneous Beatty sequences with the given densities exists. Fraenkel's Conjecture has been proven for $k = 3,4,5,6,7$ (see Altman \cite{Altman2000}, Tijdeman \cite{Tijdeman2000, Tijdeman2000Exact}, Morikawa \cite{Morikawa1982, Morikawa1985}, and Bar{\'a}t \& Varj{\'u} \cite{Barat2003}), but the general case remains open.
    
    In \cite{beatty-paper}, we considered partitions of $\NN$ into three Beatty sequences or small perturbations of Beatty sequences and proved the following theorem.
    \begin{thm}[{\cite[Theorem 4]{beatty-paper}}]
    \label{thm:BeattyPartition}
    Given positive irrational numbers $\alpha,\beta, \gamma$ such that 
    \[
    \max(\alpha, \beta) < \gamma \text{ and } \alpha + \beta + \gamma = 1,
    \]
    there exists a partition of $\NN$ into sequences $\Ba, \Bbt, \Bct$, where $\Ba$ is an exact Beatty sequence,  $\Bbt$ is obtained from the Beatty sequence  $\Bb$ by perturbing some elements by at most 1, and $\Bct$ is obtained from the Beatty sequence  $\Bc$ by perturbing some elements by at most 2.
    \end{thm}
    The perturbation bounds 1 and 2 in this result are best-possible \cite[Theorem 8]{beatty-paper}. 
    
    In this paper, we consider the analogous question for Webster sequences and prove an analog of Theorem \ref{thm:BeattyPartition}. Our main result, Theorem \ref{thm:main}, shows that when using Webster sequences, perturbations of at most 1 are needed to obtain a three-part partition of $\NN$. Thus, Webster sequences are more efficient in generating partitions than Beatty sequences which require perturbations of up to 2 to obtain a partition. 
    
    We also provide two efficient algorithms, Algorithm \ref{alg:1} and Algorithm \ref{alg:2}, to generate the partitions. The first algorithm outputs the $n$th element of each of the three sequences and the second algorithm outputs the assignment of the $m$th natural number to a sequence, both in $O(1)$ time. 

    Our results have natural interpretations in the context of apportionment problems and Just-In-Time sequencing problems. In the following sections, we describe some of these applications.

\subsection{Apportionment Methods}
\label{subsec:ApportionMethods}
Suppose we have a union of $k$ states, with relative populations $\alpha_1,\dots, \alpha_k$, respectively (so that $\sum_{1\leq i\leq k} \alpha_i =1$). We want to distribute $n$ seats in the house to $n$ representatives selected from these $k$ states. Denote by $s_i$ the number of representatives selected from state $i$. The numbers $s_i$ must be nonnegative integers and satisfy $\sum_{1\leq i\leq k} s_i = n$. This problem is called the \emph{Discrete Apportionment Problem}. It has its origins in the problem of seat assignments to the house of representatives in the United States. The goal is to assign the seats in the fairest possible manner. While the ``fairness" can be measured by various objective functions, most measures in the literature involve the differences between the fair share of representatives and the actual count of assigned representatives for each of the $k$ states (i.e., the quantity $|s_i - n\alpha_i|$). Some desirable properties of an apportionment method are the following. Details and more properties can be found in Balinski \& Young \cite[Appendix A]{Balinski-Young1982}, Kubiak \cite[Chapter 2.3]{Kubiak2008}, Pukelsheim \cite{Pukelsheim2017}, Jozefowska \cite[Chapter 2]{Jozefowska2007-book}, or Thapa \cite[Section 2.3]{Thapa2015}.

\emph{House Monotone}: An apportionment method should have the property that the number of seats of each state is non-decreasing when $n$ increases. This prevents the Alabama paradox, where increasing the number of seats in the house can result in a state losing seats \cite{Balinski-Young1982, Janson2012, Bautista1996}.
 
\emph{Population Monotone}: For any state $i$, the
number of representatives from state $i$ should not decrease if $\alpha_i$ increases.

\emph{Quota Condition}: 
     The ``fair share" of seats for state $i$ is $n\alpha_i$. We call this number the \emph{quota} for state $i$. We say an apportionment satisfies the \emph{quota condition} if $s_i \in \{ \fl{n\alpha_i}, \cl{n\alpha_i} \}$ for each $i$, i.e., if the actual number of seats assigned is within $\pm 1$ of the fair share.

The impossibility theorem of Balinski and Young \cite{Balinski-Young1980, Balinski-Young1982} states that it is in general not possible to satisfy all the desirable requirements (including the three above) . 

An important class of apportionment methods are divisor methods, which assign seats according to the formula 
\begin{equation*}
    s_i = d\left(\frac{\alpha_i}{\lambda}\right),
\end{equation*}
where $d(x)$ is a rounding function and $\lambda$ is a divisor chosen such that $\sum_{i=1}^k s_i = n$ (see \cite[Appendix A]{Balinski-Young1982}, \cite[Section 4.5]{Pukelsheim2017}, \cite[Section 2.6]{Jozefowska2006}, or \cite[Appendix A]{Janson2014}). All divisor methods are house and population monotone \cite[Theorem 4.3 \& Corollary 4.3.1]{Balinski-Young1982}, but in general they do not satisfy the quota condition \cite[Chapter 10]{Balinski-Young1982}.

A particular divisor method is \emph{Webster's method}, which specifies $d(x) = \cl{x- 1/2}$ as the rounding function, i.e., $\alpha_i/\lambda$ is rounded to the nearest integer \cite{Balinski-Young1980}. Webster's method has several unique properties. It is the only divisor method that satisfies the quota condition for $k = 3$ \cite[Theorem 6]{Balinski-Young1980}. It is the fairest method judged by natural criteria suggested by real-life problems \cite{Balinski-Young1980}. In the case $k=2$ and irrational densities $\alpha,\beta$, Webster's method generates the partition given by Theorem \ref{thm:webster-k=2} into the sequences $\Wa$ and $\Wb$ defined in \eqref{def-webster} (see \cite[Section 5.6]{Kubiak2008}). This justifies calling these sequences \emph{Webster sequences}.



\subsection{The Chairman Assignment Problem}
\label{subsec:ChairmanAssignProb}
A related problem is the \emph{Chairman Assignment Problem} \cite{Tijdeman1980, Meijer1973, Coppersmith2011, Schneider1996, Tijdeman1973}: Suppose a union is formed by $k$ states, with associated weights $\alpha_i$ (so that $\sum_{1\leq i\leq k} \alpha_i =1$). In each year, a union chair is selected from one of these $k$ states. We try to construct an assignment of chairs such that the number of chairs selected from each state $i$ in the first $n$ years should be as close as possible to its expected count $n\alpha_i$. More precisely, we seek an assignment $\omega$ that minimizes the quantity
\begin{equation}
    \label{eq:TijdemanErrorD}
    D(\omega) =  \sup_{n\in\NN}  \sup_{i = 1, \dots, k} |A_\omega(i, n)  - n\alpha_i|,
\end{equation}
where $A_\omega(i, n)$ denotes the number of chairs selected by $\omega$ from state $i$ in the first $n$ years.

It is known \cite{Meijer1973, Tijdeman1980} that $\inf_\omega  D(\omega) = 1- 1/(2k-2)$ and Tijdeman \cite{Tijdeman1980} gave an algorithm that generates, in linear time, the assignment that achieves this bound. Schneider \cite{Schneider1996} gives an improvement of this bound in some special cases. In particular, when $k = 2$,  $\inf_\omega  D(\omega) = 1/2$, and this bound is achieved by a partition into Webster sequences. 
\subsection{Just-In-Time Sequencing}
\label{subsec:JIT}

Another closely related problem is the Just-In-Time (JIT) Sequencing Problem (see Jozefowska \cite{Jozefowska2007-book} and Kubiak \cite{Kubiak2008}). In this problem, we seek to produce $k$ types of products with demands $d_1, \dots, d_k \in \NN$, respectively, such that $\sum_{i=1}^k d_i = N$. Assume that it takes one unit of time to produce one product of any kind. Then, the production rate for product $i$ is $\alpha_i = d_i/N$ and the ``ideal" production for product $i$ at time $n$ is $n\alpha_i$. To measure the deviation of the actual productions of products from their ideal productions, Miltenburg \cite{Miltenburg1989} proposed the function
    \begin{equation}
        \label{def:deviation-max}
        \max_{n=1}^N \max_{i=1}^k  |a_{in} - n \alpha_i|,
    \end{equation}
    where $a_{in}$ denotes the actual production of product $i$ in time $n$. Note that this is an analog of \eqref{def:deviation-delta} and \eqref{eq:TijdemanErrorD}. Other authors \cite{Dhamala-Thapa2012, Miltenburg-Steiner-Yeomans1990, Inman-Bulfin1991, Miltenburg-Sinnamon1989, Bautista1996}  considered various types of average deviations such as
   
    \begin{equation}
        \label{def:deviation-sum}
       \sum_{n=1}^N \sum_{i=1}^k  (a_{in} - n \alpha_i)^2
    \end{equation}
    and
    \begin{equation}
        \label{def:deviation-InmanBulfin}
        \sum_{i=1}^k \sum_{n=1}^{d_i} \left(y_{in} -  \frac{n-1/2}{\alpha_i}  \right)^2,
    \end{equation}
    where $y_{in}$ denotes the time when the $n$th copy of product $i$ is produced. More general measures have been proposed by Dhamala \& Kubiak \cite{Dhamala-Kubiak2005}, Kubiak \cite{Kubiak1993}, Thapa \& Silvestrov \cite{Thapa-Silvestrov2015}, and Steiner \& Yeomans \cite{Steiner-Yeomans1993}. Efficient algorithms to minimize the different types of deviations defined above  have been given by Kubiak \& Sethi \cite{Kubiak-Sethi1994}, Steiner \& Yeomans \cite{Steiner-Yeomans1993}, and Inman \& Bulfin \cite{Inman-Bulfin1991}. Kubiak \cite[Theorem 3.1]{Kubiak2008} showed that, if one disregards the possibility of conflicting assignments, then assigning the $n$th element of the $i$th sequence to position $\cl{(n-1/2)/\alpha_i}$ is the optimal solution to a very general class of minimization problems. Note that $\cl{(n-1/2)/\alpha_i}$ is exactly the $n$th term of the Webster sequence $W_{\alpha i}$ defined in \eqref{def-webster}.
    
    There are close connections between the JIT sequencing problem and the apportionment problem. For example, the Inman-Bulfin algorithm for minimizing \eqref{def:deviation-InmanBulfin} is equivalent to the Webster divisor method (see \cite{Bautista1996} and \cite{Jozefowska2006}). Moreover, the quota condition described in Section \ref{subsec:ApportionMethods} is equivalent to requiring \eqref{def:deviation-max} to be bounded by 1. Also, the ``house monotone" condition in the JIT problem means that the total number of products of each type produced up to time $n$ is a monotone function of $n$ and thus is a natural requirement in this problem.
    

\subsection{Outline of Paper}
\label{subsec:Outline}
In Section \ref{sec:main}, we state our main theorem, Theorem \ref{thm:main}, which shows that, under some mild conditions on $\alpha, \beta, \gamma$, we can get a three-part partition of $\NN$ into sequences of densities $\alpha, \beta, \gamma$ by perturbing some elements of two of the Webster sequences $\Wa, \Wb, \Wc$ by at most 1. We give two equivalent algorithms, Algorithms \ref{alg:1} and  \ref{alg:2}, that explicitly generate the partition described in Theorem \ref{thm:main}. Section \ref{sec:aux-lemmas} gathers some auxiliary lemmas and Sections \ref{sec:proof-algos-equiv}-\ref{sec:proof-main-thm} 
contain the proofs of the main results. In Section \ref{sec:optimality}, we show that our results are best-possible in several key aspects. In Section \ref{sec:concluding-remarks}, we present some additional results and open problems.

\section{Main Results}
\label{sec:main}
\subsection{Notations and Terminology}
\label{subsec:notations}
We use $\NN$ to denote the set of positive integers, and we 
use capital letters $A,B,\dots$, to denote subsets of $\NN$ or,
equivalently, strictly increasing sequences of positive integers. We denote
the $n$th elements of such sequences by $a(n)$, $b(n)$, etc. and we will usually use the letter $m$ to denote the positions of these elements within $\NN$, i.e., we may write $m = b(n)$. 

We denote the floor and ceiling of a real number $x$ by
\begin{align*}
    \fl{x} &= \max\{n\in\NN: n \leq x\},\\
    \cl{x} &= \min\{n\in\NN: n \geq x\}.
\end{align*}
We denote the fractional part of $x$ by 
\[
\fp{x} = x-\fl{x}.
\]
Given a set $A\subset \NN$, we denote the counting function of A by 
\begin{equation}
\label{eq:def-counting-function}
\quad A(m)=\#\{m'\le m: m'\in A\} \quad (m\in\NN).
\end{equation}
For convenience, we also define $A(0) = 0$.
\begin{defn}[Webster sequences and almost Webster sequences]
 Let $\alpha\in(0,1)$.
 \begin{itemize}
     \item [(i)] We define the \emph{Webster sequence of density $\alpha$} as
     \begin{align}
    \label{def:def-webster-a}
    \Wa &=\seq{a(n)},\quad a(n)=\Cl{\frac{n-1/2}{\alpha}}.
    \end{align}
     \item [(ii)] We call a sequence $\Wat=\seq{\at(n)}$ an  \emph{almost Webster sequence of density $\alpha$} if, for any $n\in\NN$,  the $n$th element of $\Wat$ and $n$th element of $\Wa$ differ by at most 1. That is, a sequence $\Wat$ is an \emph{almost Webster sequence of density $\alpha$} if
    \begin{equation}
    \label{def:almost-webster}
        |\at(n)-a(n)| \leq 1 \quad (n\in\NN).
    \end{equation}
 \end{itemize}
\end{defn}

Note that any Webster sequence is trivially an almost Webster sequence of the same density.


\subsection{Statement of Main Theorem}
\label{subsec:results}

\begin{thm}[Partition into one exact Webster sequence and two almost
Webster sequences]
\label{thm:main}
Let $\alpha,\beta,\gamma$ satisfy
\begin{equation}
\label{eq:abc-condition}
\alpha,\beta,\gamma\in \RR^+\setminus\QQ,\quad
\alpha+\beta+\gamma=1, \quad\beta<1/2,\quad \alpha<\gamma. 
\end{equation}
Then there exists a partition of $\NN$ into a Webster sequence, $\Wa$, of density $\alpha$ and two almost Webster sequences, $\Wbt$ and $\Wct$, of densities $\beta$ and $\gamma$, respectively. That is, there exist sequences $\Wat=\seq{\at(n)}$, $\Wbt=\seq{\bt(n)}$, and $\Wct=\seq{\ct(n)}$ that partition $\NN$ and satisfy
\begin{align}
\label{eq:mainThm-bounds-a}
\at(n)&=a(n)
\quad (n\in\NN),
\\
\label{eq:mainThm-bounds-b}
\bt(n)-b(n)&\in\{-1,0,1\}
\quad (n\in\NN),
\\
\label{eq:mainThm-bounds-c}
\ct(n)-c(n)&\in\{-1,0,1\}
\quad (n\in\NN),
\end{align}
where $\seq{a(n)}, \seq{b(n)}, \seq{c(n)}$ are the Webster sequences of densities $\alpha,\beta,\gamma$, respectively.
\end{thm}
In Section \ref{sec:optimality}, we will show that this result is best-possible in the sense that the conditions in \eqref{eq:mainThm-bounds-b} and \eqref{eq:mainThm-bounds-c} cannot be replaced by stronger conditions (see Theorem \ref{thm:optimality:one-sided-impossible}). In particular, it is, in general, not possible to obtain a partition of $\NN$ into two exact Webster sequences and one almost Webster sequence with prescribed densities (see Theorem \ref{thm:optimality:TwoExactOneAlmost}).

\subsection{Partition Algorithms}
\label{subsec:algos}
In the following, we provide two equivalent algorithms that generate the partition in Theorem \ref{thm:main}.
We assume that $\alpha, \beta, \gamma$ satisfy condition \eqref{eq:abc-condition}. Note that condition \eqref{eq:abc-condition} implies
\begin{equation}
\label{eq:ab-less-half}
    \alpha,\beta < 1/2 \text{ and } \gamma > 1/4
\end{equation}
and 
\begin{equation}
    \label{eqproof:a+b/2<1/2}
    \alpha + \beta/2 < 1/2 <   \gamma + \beta/2.
\end{equation}

We introduce the following notation that will appear frequently in the remainder of this paper:
\begin{equation}
\label{def:uv}
    u_m = \fp{m\alpha+1/2},\quad v_m = \fp{m\beta+1/2},
\end{equation}
where $\fp{x}$ denotes the fractional part of $x$. Note that our assumption $\alpha,\beta \notin \QQ$ implies
\begin{equation}
    \label{eq:uvequality-impossible}
    u_m \notin \alpha\QQ+\ZZ, \quad  v_m \notin \beta\QQ+\ZZ. 
\end{equation}
To prove \eqref{eq:uvequality-impossible}, suppose $u_m \in \alpha\QQ + \ZZ$. Then $m\alpha + 1/2 = \alpha q + k$ for some $q\in\QQ$ and $k\in \ZZ$. If $q = m$, then $1/2 = k$, which contradicts $k\in\ZZ$. If $q\neq m$, then $\alpha = (k-1/2) / (m-q)$, which contradicts $\alpha \notin \QQ$.

The partition algorithms receive inputs $\alpha, \beta, \gamma$, and output almost Webster sequences $\Wat, \Wbt,\Wct$ that partition $\NN$. Our first algorithm generates the $n$th element in each of the sequences $\Wat, \Wbt, \Wct$.
\begin{alg}
\label{alg:1}
Let  $\Wat =\seq{\at(n)}$ and $\Wbt=\seq{\bt(n)}$ be defined by
\begin{align}
\label{eq:thm-at-def}
&\at(n) =  a(n),\\ 
\label{eq:thm-bt-def}
&\bt(n)=\begin{cases}
b(n)
&\text{if $u_m > \alpha$;}
\\
b(n)-1 &\text{if $u_m  < \alpha$ and $ v_m > \beta/2$;}
\\
b(n)+1 &\text{if $u_m  < \alpha$ and $v_m < \beta/2$,}
\end{cases}
\end{align}
where $m = b(n)$, and let $\Wct=\NN\setminus(\Wa\cup \Wbt)=\seq{\ct(n)}$. 
\end{alg}
Note that the three conditions in  \eqref{eq:thm-bt-def} are mutually exclusive and exhaust all possibilities, since by \eqref{eq:uvequality-impossible}, the cases $u_m=\alpha$ and $v_m = \beta/2$ cannot occur. Moreover, by \eqref{eq:thm-at-def} and \eqref{eq:thm-bt-def}, the sequences $\seq{\at(n)}$ and $\seq{\bt(n)}$ generated by Algorithm \ref{alg:1} clearly satisfy the conditions \eqref{eq:mainThm-bounds-a} and \eqref{eq:mainThm-bounds-b} of Theorem \ref{thm:main} and thus, in particular, are almost Webster sequences. In Proposition \ref{prop:error-diff-c} below, we will show that the sequence $\seq{\ct(n)}$ satisfies condition \eqref{eq:mainThm-bounds-c} and thus is also an almost Webster sequence. 

We remark that Algorithm \ref{alg:1} gives the $n$th element of the sequences $\Wa, \Wbt, \Wct$ in $O(1)$ time. This is clear from \eqref{eq:thm-at-def} and \eqref{eq:thm-bt-def} in the case of the sequences $\Wa$ and $\Wbt$. In the case of the sequence $\Wct$, a similar characterization of $\ct(n)$ proved in Proposition \ref{prop:error-diff-c} below allows one to compute the $n$th term of $\Wct$ in $O(1)$ time. By contrast, most algorithms in the literature on scheduling problems are recursive and therefore have run-time at least $O(n)$.

Our second algorithm outputs, for a given $m\in\NN$, the sequence to which $m$ is assigned. 
\begin{alg}
\label{alg:2}
For any $m \in \NN$, let
\begin{equation}
\label{eq:alg2-m}
m\in \begin{cases}
\Wat
&\text{if $u_m < \alpha$;}
\\
\Wbt&\text{if ($u_m > \alpha $ and $v_m<\beta$) or ($u_m >1-\alpha $ and $v_m>1-\beta/2$)}\\
&\text{\quad or ($\alpha<u_m <2\alpha $ and $\beta<v_m<3\beta/2$);}
\\
\Wct&\text{otherwise.}
\end{cases}
\end{equation}
\end{alg}

Again, noting that $\alpha < 1/2$ (see \eqref{eq:ab-less-half}), the three cases in \eqref{eq:alg2-m} are mutually exclusive, and exhaust all possibilities, so the algorithm indeed generates a partition of $\NN$.

In Section \ref{sec:proof-algos-equiv}, we will prove that the two algorithms are equivalent.
\begin{prop}
\label{prop:alg-equivalent}
Algorithm \ref{alg:1} and Algorithm \ref{alg:2} generate the same sequences.
\end{prop}

In Section \ref{sec:proof-counting-prop}, we will show that the counting functions for the sequences generated by Algorithm \ref{alg:2} differ from the counting functions for the corresponding Webster sequences by at most 1. This result will be a key step to proving Theorem \ref{thm:main}. 
\begin{prop}
\label{prop:countingError}
 The sequences $\Wat, \Wbt, \Wct$ constructed in Algorithm \ref{alg:2} satisfy 
 \begin{align}
    \label{eq:countingError-a}
     &\Wat(m) = \Wa(m) \quad(m\in\NN), \\
     \label{eq:countingError-b}
    &\Wbt(m)-\Wb(m)\in \{-1,0,1\} \quad(m\in\NN), \\
    \label{eq:countingError-c}
    &\Wct(m)-\Wc(m)\in \{-1,0,1\} \quad(m\in\NN).
\end{align}
\end{prop}
\begin{remark}
\label{remark:quota} 
It is not difficult to deduce from this result and Propositions \ref{prop:error-counting-b} and \ref{prop:error-counting-c} below that the sequences $\Wat$, $\Wbt$, and $\Wct$ satisfy the quota condition, i.e., $\Wat(m) \in \{\fl{m\alpha}, \cl{m\alpha}\}$, $\Wbt(m) \in \{\fl{m\beta}, \cl{m\beta}\}$ and $\Wct(m) \in \{\fl{m\gamma}, \cl{m\gamma}\}$. In the case of $\Wat$, this is obvious since  $\Wat(m) = \Wa(m) = \fl{m\alpha +1/2}$ (see Lemma \ref{lem:general-lemmas}(ii) below). For $\Wbt$ and $\Wct$, this can be derived from the explicit characterizations of the errors $-1,0,1$ in \eqref{eq:countingError-b} and \eqref{eq:countingError-c} given in Propositions \ref{prop:error-counting-b} and \ref{prop:error-counting-c} below.
\end{remark}

\begin{remark}
Note that \eqref{eq:countingError-b} and \eqref{eq:mainThm-bounds-b} represent different ways to measure the deviation of a sequence $\Wbt$ from the corresponding Webster sequence $\Wb$. In \eqref{eq:mainThm-bounds-b}, the $n$th terms of the two sequences are compared, while in \eqref{eq:countingError-b}, the counting functions of the sequences are compared. In fact, \eqref{eq:mainThm-bounds-b} implies \eqref{eq:countingError-b}, but \eqref{eq:countingError-b} does not imply \eqref{eq:mainThm-bounds-b}. For example, for sequences $B = \seq{b(n)} = \seq{nk}$ and $\widetilde{B} = \seq{\bt(n)} =  \seq{nk + k/2}$ for a fixed even integer $k > 2$, we have $\widetilde{B}(m) - B(m) \in \{-1, 0, 1\}$ for all $m$, but $\widetilde{b}(n) - b(n) = k/2 \notin \{-1, 0, 1\}$. Thus, Proposition \ref{prop:countingError} can be regarded as a weaker version of Theorem \ref{thm:main}.

\end{remark}

\section{Auxiliary Lemmas}
\label{sec:aux-lemmas}

In this section, we provide some useful identities regarding the floor and fractional part functions and some elementary properties of Webster sequences. Most of these results are analogous to results proved in \cite[Section 3]{beatty-paper}.



\begin{lem}[Generalized Weyl's Theorem]
\label{lem:weyl-general}
Let $\theta_1,\dots,\theta_k$ be real numbers such that 
the numbers $1,\theta_1,\dots,\theta_k$ are linearly independent over
$\QQ$. Let $\eta_1, \dots, \eta_k$ be arbitrary real numbers.
Then the $k$-dimensional sequence $(\fp{n\theta_1+\eta_1}, \dots, \fp{n\theta_k+\eta_k)})$
is \emph{uniformly distributed modulo $1$ in $\RR^k$}. That is,  
we have
\[
\lim_{N\to\infty}
\frac1N\#\{n\le N: \fp{n\theta_i+\eta_i}<t_i\text{ for $i=1,\dots,k$}\}
=t_1\cdots t_k\quad (0\le t_i\le 1).
\]
\end{lem}

\begin{proof}
This is a special case of Theorem 6.3 in Chapter 1 of Kuipers and Niederreiter \cite{Kuipers-Niederreiter2012} (see Example 6.1).
\end{proof}

\begin{cor}[Uniform distribution of $\fp{n\alpha+1/2}$]
\label{cor:unif-a}
Let $\alpha$ be an irrational number. Then the sequence 
\[
\fp{n\alpha+1/2},\quad n=1,2,3,\dots
\]
is uniformly distributed on the unit interval $(0,1)$.
\end{cor}
\begin{proof}
Apply Lemma \ref{lem:weyl-general} with $\theta_1 = \alpha, \eta_1 = 1/2$.
\end{proof}

\begin{cor}[Uniform distribution of pairs $(\fp{n\alpha+1/2}, \fp{n\beta+1/2})$] 
\label{cor:unif-ab}
Let $\alpha,
\beta$ be irrational numbers such that $1,\alpha, \beta$ are independent over $\QQ$. Then the pairs 
\[
(\fp{n\alpha+1/2}, \fp{n\beta+1/2}),\quad n=1,2,3,\dots
\]
are uniformly distributed on the unit square $(0,1)\times(0,1)$.
\end{cor}
\begin{proof}
Apply Lemma \ref{lem:weyl-general} with  $\theta_1 = \alpha, \theta_2 = \beta, \eta_1 = \eta_2 = 1/2$.
\end{proof}

\begin{lem}[Elementary properties of Webster sequences (cf. {\cite[Lemma 10]{beatty-paper}})]
\label{lem:general-lemmas}
Let $\alpha\in(0,1) \setminus 
\QQ$, and let 
$\Wa = \seq{a(n)} = \seq{\lceil(n-1/2)/\alpha\rceil} = \seq{\fl{(n-1/2)/\alpha}+1}$ be the Webster sequence of density $\alpha$. 
\begin{itemize}
\item[(i)] \textbf{Membership criterion:}
For any $m\in\NN$ we have
\begin{equation*}
m\in\Wa
\Longleftrightarrow \fp{m\alpha+1/2}<\alpha.
\end{equation*}
\item[(ii)] \textbf{Counting function formula:}
For any $m\in\NN$ we have
\begin{equation*}
\Wa(m)=\fl{m\alpha+1/2},
\end{equation*}
where $\Wa(m)$ is the counting function of $\Wa$.
\item[(iii)]  \textbf{Gap criterion:}
Given $m\in \Wa$, let
$m'$ denote the successor to $m$ in the sequence 
$\Wa$. Let $k=\fl{1/\alpha}$. Then, $m'=m+k$ or $m'=m+k+1$. Moreover, for any $m\in\NN$,
\begin{align}
\label{eq:gap-criterion-1}
m\in\Wa\text{ and } m'=m+k
&\Longleftrightarrow 
\fp{1/\alpha}\alpha< \fp{m\alpha+1/2}<\alpha;
\\
\label{eq:gap-criterion-2}
m\in\Wa\text{ and } m'=m+k+1
&\Longleftrightarrow 
\fp{m\alpha+1/2}<\fp{1/\alpha}\alpha.
\end{align}
\end{itemize}
\end{lem}

Note that by \eqref{eq:uvequality-impossible}, equality cannot hold in \eqref{eq:gap-criterion-1} and \eqref{eq:gap-criterion-2}.

\begin{proof}
(i) We have 
\begin{align*}
m\in\Wa&\Longleftrightarrow
m=\fl{(n-1/2)/\alpha+1}\text{ for some $n\in\NN$}
\\
&\Longleftrightarrow  m< (n-1/2)/\alpha+1<m+1\text{ for some $n\in\NN$}
\\
&\Longleftrightarrow  (m-1)\alpha< (n-1/2)<m\alpha\text{ for some $n\in\NN$}
\\
&\Longleftrightarrow  (m-1)\alpha+1/2< n<m\alpha+1/2\text{ for some $n\in\NN$}
\\
&\Longleftrightarrow \fp{m\alpha+1/2} < \alpha.
\end{align*}
(ii) We have
\begin{align*}
\Wa(m)=n&\Longleftrightarrow
\fl{(n-1/2)/\alpha}+1\le m\le (\fl{(n+1/2)/\alpha}+1)-1
\\
&\Longleftrightarrow
(n-1/2)/\alpha+1< m+1< (n+1/2)/\alpha+1
\\
&\Longleftrightarrow
(n-1/2)< m\alpha< (n+1/2)
\\
&\Longleftrightarrow
n< m\alpha+1/2< n+1
\\
&\Longleftrightarrow
\fl{m\alpha+1/2} =n.
\end{align*}
(iii)
Suppose $m = a(n)$ and $m' = a(n+1)$ are successive elements in $\Wa$. We have

\begin{align*}
    m'-m &= a(n+1)-a(n)=\Fl{\frac{n+1-1/2}{\alpha}}-\Fl{\frac{n-1/2}{\alpha}}\\
    &=\Fl{\frac1{\alpha}} \text{ or } \Cl{\frac1{\alpha}}.
\end{align*}
Hence, $m'$ must be equal to $m+k$ or $m+k+1$, where $k = \fl{1/\alpha}$. 
Note that
\[
k\alpha=\left(\frac{1}{\alpha}-\Fp{\frac{1}{\alpha}}\right)\alpha
=1-\Fp{\frac1{\alpha}}\alpha.
\]
Thus, using the results of part (i) we have
\begin{align*}
m\in\Wa\text{ and } m'=m+k&\Longleftrightarrow
m\in\Wa\text{ and } m+k\in\Wa
\\
&\Longleftrightarrow
\fp{m\alpha+1/2}<\alpha\text{ and } 
\fp{m\alpha+1/2+k\alpha}<\alpha
\\
&\Longleftrightarrow
\fp{m\alpha+1/2}<\alpha\text{ and } 
\fp{m\alpha+1/2-\fp{1/\alpha}\alpha}<\alpha
\\
&\Longleftrightarrow
\fp{1/\alpha}\alpha<\fp{m\alpha+1/2}<\alpha.
\end{align*}
This proves the equivalence \eqref{eq:gap-criterion-1} in (iii). Observe that, by (i), $m\in\Wa$ is equivalent to
$\fp{m\alpha+1/2}\in(0,\alpha)$. The 
equivalence \eqref{eq:gap-criterion-2} in (iii) then follows from the fact that $m'$ is either $m+k$ or  $m+k+1$.
\end{proof}

\section{Proof of Proposition \ref{prop:alg-equivalent}}
\label{sec:proof-algos-equiv}
We assume $\alpha,\beta,\gamma$ satisfying condition \eqref{eq:abc-condition}. For any $m\in\NN$, define $u_m,v_m$ as in \eqref{def:uv} (i.e., $u_m = \fp{m\alpha+1/2}$ and $v_m = \fp{m\beta+1/2}$).

\begin{proof}[Proof of Proposition  \ref{prop:alg-equivalent}]
We will prove that Algorithm \ref{alg:1} is equivalent to Algorithm \ref{alg:2}. By Lemma \ref{lem:general-lemmas}(i) and the definition of $u_m$ we have
 \begin{equation*}
     m \in \Wa \Longleftrightarrow u_m < \alpha.
 \end{equation*}
  Thus, the sequence $\Wat$ defined in Algorithm \ref{alg:2} is the Webster sequence $\Wa$. Moreover, by definition, the sequence $\Wat$ in Algorithm \ref{alg:1} is also the Webster sequence $\Wa$. Therefore, the sequences $\Wat$ generated by the two algorithms are the same. Also, in both algorithms, the third sequence, $\Wct$, is defined as $\NN \setminus (\Wat\cup\Wbt)$. Therefore, it remains to show that the sequences $\Wbt$ constructed by the two algorithms are the same.
 
 Note that, if $m = b(n)$, we have $m \in \Wb$ and therefore, by Lemma \ref{lem:general-lemmas}(i), we necessarily have $v_m < \beta$. Thus, in the first case in \eqref{eq:thm-bt-def}, we can replace the condition ``$u_m > \alpha$" by ``$u_m > \alpha \text{ and } v_m < \beta$", and setting $\widetilde{m} = \bt(n)$, we can write the three cases in \eqref{eq:thm-bt-def} as
\begin{equation}
    \label{eqproof:assign-Wbt}
    \begin{cases}
        \text{(I)} &u_{\widetilde{m}} > \alpha \text{ and }v_{\widetilde{m}}<\beta,\\
    \text{(II)} & u_{\widetilde{m}+1} < \alpha \text{ and }\beta/2 < v_{\widetilde{m}+1}<\beta,\\
      \text{(III)} &u_{\widetilde{m}-1} < \alpha  \text{ and } v_{\widetilde{m}-1}<\beta/2.
    \end{cases}
\end{equation}
Using the elementary relation
\begin{equation*}
  u_{\widetilde{m}+1} = \fl{(\widetilde{m}+1)\alpha+1/2} = \fl{\widetilde{m}\alpha+1/2+\alpha} =
 \begin{cases}
    u_{\widetilde{m}} +\alpha -1 \quad &\text{if $u_{\widetilde{m}} > 1-\alpha$};\\
      u_{\widetilde{m}} +\alpha \quad &\text{if $u_{\widetilde{m}} < 1-\alpha$},
\end{cases}
\end{equation*}
and a similar relation  between $u_{\widetilde{m}-1}$ and $u_{\widetilde{m}}$, one can check that the cases (I), (II), and (III) in \eqref{eqproof:assign-Wbt} are equivalent, respectively, to the three cases for $\widetilde{m}\in\Wbt$ in Algorithm \ref{alg:2}, i.e., to
\begin{equation*}
\begin{cases}
    \text{(I')} &u_{\widetilde{m}} > \alpha \text{ and }v_{\widetilde{m}}<\beta,\\
    \text{(II')} & u_{\widetilde{m}} >1-\alpha \text{ and }v_{\widetilde{m}}>1-\beta/2,\\
     \text{(III')} & \alpha<u_{\widetilde{m}} <2\alpha \text{ and }\beta<v_{\widetilde{m}}<3\beta/2.
    \end{cases}
\end{equation*}
 The equivalence between Algorithm \ref{alg:1} and Algorithm \ref{alg:2} follows.
\end{proof}

\section{Proof of Proposition \ref{prop:countingError}}
\label{sec:proof-counting-prop}
 We fix real numbers $\alpha,\beta, \gamma$ that satisfy condition $\eqref{eq:abc-condition}$ and assume $\Wa, \Wbt, \Wct$ are the three sequences generated by Algorithm \ref{alg:1} that give a partition of $\NN$ . 

Let $u_m,v_m$ be defined by \eqref{def:uv}, (i.e., $u_m = \fp{m\alpha+1/2}$ and $v_m = \fp{m\beta+1/2}$) and define $w_m$ analogously with respect to $\gamma$, i.e.,
\begin{equation}
\label{def:w}
    w_m = \fp{m\gamma+1/2}.
\end{equation}
Note that
\begin{align}
    \label{eq:w-further}
    w_m &= \fp{m(1-\alpha-\beta)+1/2}\\\notag
    &=  \fp{-m(\alpha+\beta)+1/2}\\\notag
     &=  1 - \fp{m(\alpha+\beta)-1/2}\\\notag
     &= 1-\fp{u_m+v_m-1/2}.
\end{align}

We define
\begin{align}
    \label{def:ebt}
    &\Ebt(m) = \Wbt(m)-\Wb(m),\\
    \label{def:ect}
    &\Ect(m) = \Wct(m)-\Wc(m).
\end{align}
Then Proposition \ref{prop:countingError} follows from Propositions \ref{prop:error-counting-b} and \ref{prop:error-counting-c} below, which show that $\Ebt(m), \Ect(m) \in \{-1,0,1\}$ and characterize the values of $\Ebt(m)$ and $\Ect(m)$ in terms of the numbers $u_m$ and $v_m$.
\begin{prop}
\label{prop:error-counting-b}
For any $ m \in \mathbb{N}$, $\Ebt(m) \in \{1, -1, 0\}$ and
\begin{equation}
\label{eq:error-counting-b}
\Ebt(m)= 
\begin{cases}
1 &\text{if $u_m > 1-\alpha$ and  $v_m > 1-\beta/2$};\\
-1 &\text{if $u_m < \alpha$ and $v_m < \beta/2$};\\
0 &\text{otherwise}.
\end{cases}
\end{equation}
\end{prop}

\begin{prop}
\label{prop:error-counting-c}
For any $ m \in \mathbb{N}$, $\Ect(m) \in \{1, -1, 0\}$ and
\begin{equation}
\label{eq:error-counting-c}
\Ect(m)= 
\begin{cases}
1 &\text{if $3/2 < u_m+v_m <2$ and ($ u_m < 1-\alpha$ or $v_m < 1-\beta/2$)};\\
-1 &\text{if  $0 < u_m+v_m < 1/2$ and ( $u_m > \alpha$ or $v_m > \beta/2$)};\\
0 &\text{otherwise}.
\end{cases}
\end{equation}
\end{prop}

\subsection{Proof of Proposition \ref{prop:error-counting-b}
}
By Lemma \ref{lem:general-lemmas}(i),  \eqref{eq:thm-bt-def} is equivalent to
\begin{equation}
\label{eq:thm-bt-def-equiv}
\bt(n)=\begin{cases}
b(n)
&\text{if $b(n) \in \Wb \setminus\Wa$;}
\\
b(n)-1 &\text{if $b(n) \in \Wa\cap\Wb$ and $v_{b(n)} > \beta/2$;}
\\
b(n)+1 &\text{if $b(n) \in \Wa\cap\Wb$ and $v_{b(n)} < \beta/2$}
\end{cases}
\end{equation}
By the construction of $\Wbt$ in \eqref{eq:thm-bt-def}, for all $m$ such that $m+1, m, m-1 \notin \Wb$, we have $\Wbt(m) = \Wb(m)$ and thus $\Ebt(m) = 0$. Therefore, it suffices to consider the following cases:
\[
(I)\ m+1 \in \Wb, \quad (II)\ m \in \Wb,\quad (III)\ m-1 \in \Wb.
\]

\textit{Case I: $m+1 \in \Wb$.} Then $m+1 = b(n)$ for some $n$. By Lemma \ref{lem:general-lemmas}(iii), the gap between any two successive elements in $\Wb$ is at least 2, since $\beta < 1/2$. If $b(n) \notin \Wa$, then $\bt(n) = b(n)$ by \eqref{eq:thm-bt-def-equiv}. Therefore, $\Wbt(m) = \Wb(m) = n-1$ and thus $\Ebt(m) = 0$. If $b(n) \in \Wa$, then $\Ebt(m)$ is nonzero only if $\bt(n) = b(n)-1$, in which case $\Ebt(m) = \Wbt(m) - \Wb(m) = n - (n-1) = 1$. By \eqref{eq:thm-bt-def-equiv}, $\bt(n) = b(n)-1$ if and only if 
\[
m+1 \in \Wa\cap\Wb\text{ and } v_{m+1} > \frac{\beta}{2},
\]
which, by Lemmas \ref{lem:general-lemmas}(i), is equivalent to 
\[
u_m > 1-\alpha \text{ and } v_m > 1-\frac{\beta}{2}.
\]

\textit{Case II: $m \in \Wb$.} Then $m = b(n)$ for some $n$. Similarly as in Case I, $\Ebt(m)$ is nonzero only if $\bt(n) = b(n)+1$, in which case $\Ebt(m) = -1$. By \eqref{eq:thm-bt-def-equiv}, $\bt(n) = b(n)+1$ holds if and only if 
\[
m \in \Wa\cap\Wb\text{ and }v_{m} < \frac{\beta}{2},
\]
which is equivalent to 
\[
u_m < \alpha \text{ and } v_{m} < \frac{\beta}{2}.
\]

\textit{Case III: $m-1 \in \Wb$.} Then $m-1 = b(n)$ for some $n$. By Cases I and II, we can assume $m+1 \notin \Wb $ and $m \notin \Wb$. Then $b(n-1) \leq m-2$ and $b(n+1) \geq m+2$, so $\bt(n-1) \leq b(n-1) + 1 \leq m-1$ and $\bt(n+1) \geq b(n+1) - 1 \geq m+1$. Therefore $\Wbt(m) = n = \Wb(m)$ and hence $\Ebt(m) = 0$.

 Combining the three cases gives \eqref{eq:error-counting-b} and thus $\Ebt(m) \in \{1, -1, 0\}$ for any $ m \in \mathbb{N}$. \qed

\subsection{Proof of Proposition \ref{prop:error-counting-c}
}
\label{subsec:proof-prop-c-sec5}

\begin{lem}
\label{lem:three-counting-sum}
We have, for any $m \in \NN,$
\begin{equation*}
\label{eq:three-counting-sum}
    W_\alpha(m) + W_\beta(m) + W_\gamma(m) = m - \fl{u_m+v_m-1/2}.
\end{equation*}

\begin{proof}
Using equation \eqref{eq:w-further}, we get
\begin{align*}
     W_\alpha(m) + &W_\beta(m) + W_\gamma(m)\\
    &= \fl{m\alpha + 1/2} + \fl{m\beta + 1/2} + \fl{m\gamma + 1/2}\\
    &= m + 3/2 - u_m - v_m - w_m\\
    &= m + 3/2 - u_m - v_m - (1-\fp{u_m+v_m-1/2})\\
     &= m - \fl{u_m+v_m-1/2}. \qedhere
\end{align*}
\end{proof}
\end{lem}

\begin{proof}[Proof of Proposition \ref{prop:error-counting-c}]
From the construction of the algorithms, it is clearly that $\Wa, \Wbt, \Wct$ partition $\NN$. Therefore, for any $m\in\NN$, we have
\begin{equation}
    \label{eqproof:Wct-original}
     \Wct(m)=m-\Wa(m)-\Wbt(m)=m-\Wa(m)-\Wb(m)-\Ebt(m).
\end{equation}
By Lemma \ref{lem:three-counting-sum}, 
\begin{equation}
    \label{eqproof:Wc-original}
    W_\gamma(m) = m -W_\alpha(m)- W_\beta(m)
     - \fl{u_m+v_m-1/2}.
\end{equation}
Subtract \eqref{eqproof:Wc-original} from \eqref{eqproof:Wct-original} to get
\begin{equation}
    \label{eqproof:E=delta-delta-E}
    \Ect(m)=\Wct(m)-\Wc(m)= \fl{u_m+v_m-1/2} -\Ebt(m).
\end{equation}

We obtain, by Proposition \ref{prop:error-counting-b},
\begin{align}
\label{eqproof:delta2}
   \Ect(m) 
    &=\begin{cases}
    \text{$\fl{u_m+v_m-3/2}$\quad if $u_m > 1-\alpha$ and  $v_m > 1-\beta/2$,}\\
     \text{$\fl{u_m+v_m+1/2}$\quad if $u_m < \alpha$ and $v_m < \beta/2$,}\\
     \text{$\fl{u_m+v_m-1/2}$\quad otherwise.}
    \end{cases}
\end{align}
Note that $\alpha+\beta/2 < (\alpha+\beta+\gamma)/2 = 1/2$. Therefore, 
\begin{align}
\notag
\begin{cases}
    \text{if $u_m > 1-\alpha$ and  $v_m > 1-\beta/2$, then  $\Ect(m) = 0$;}\\
    \text{if $u_m < \alpha$ and $v_m < \beta/2$, then  $\Ect(m) = 0$;}\\
    \text{if $u_m+v_m < 1/2$ and ($u_m > \alpha$ or $v_m > \beta/2$), then  $\Ect(m) = -1$;}\\
    \text{if $u_m+v_m > 3/2$ and ($u_m < 1-\alpha$ or $v_m < 1-\beta/2$), then  $\Ect(m) = 1$;}\\
    \text{otherwise, $\Ect(m) = 0$.}\\
    \end{cases}
\end{align}
Thus,  Proposition \ref{prop:error-counting-c} holds.
\end{proof}

\section{Proof of Theorem
\ref{thm:main}}
\label{sec:proof-main-thm}
 We fix real numbers $\alpha,\beta, \gamma$ that satisfy condition $\eqref{eq:abc-condition}$ and assume $\at(n), \bt(n), \ct(n)$ are the three sequences generated by Algorithm \ref{alg:1} that give a partition of $\NN$.
 
 By construction, the sequences $\at(n)$ and $\bt(n)$ satisfy the conditions \eqref{eq:mainThm-bounds-a} and \eqref{eq:mainThm-bounds-b} of Theorem \ref{thm:main}. Thus it remains to show that $\ct(n)$ satisfies \eqref{eq:mainThm-bounds-c}, i.e., $\ct(n)-c(n) \in \{-1,0,1\}$. This follows from Proposition \ref{prop:error-diff-c} below, which gives the desired error bounds for $\ct(n)$ and also provides a complete characterization for the errors $-1, 0, \text{ and } 1$, respectively.


\begin{prop}
\label{prop:error-diff-c}
Given $n\in\mathbb{N}$, set $m = c(n), \widetilde{m}=\ct(n)$. Then $m-\widetilde{m} \in \{-1, 0, 1\}$ and
\begin{equation}
\label{eq:error-diff-c}
m-\widetilde{m}=
\begin{cases}
-1 &\text{if $u_{\widetilde{m}} > \alpha$ and $v_{\widetilde{m}} > \beta$ and ($u_{\widetilde{m}} >2\alpha$ or $v_{\widetilde{m}} > 3\beta/2$)}\\
&\text{and $u_{\widetilde{m}}+v_{\widetilde{m}} < 3/2-\gamma$};\\
0&\text{if $u_{\widetilde{m}}>\alpha$ and $v_{\widetilde{m}}>\beta$ and ($3/2-\gamma < u_{\widetilde{m}}+v_{\widetilde{m}} < 3/2$)};\\
1&\text{if $3/2 < u_{\widetilde{m}}+v_{\widetilde{m}} <2$ and ($u_{\widetilde{m}} < 1-\alpha$ or $v_{\widetilde{m}} < 1-\beta/2$)}.
\end{cases}
\end{equation}
\end{prop}
The distribution of $m-\widetilde{m}$, given by \eqref{eq:error-diff-c}, in terms of the pair $(u_{\widetilde{m}}, v_{\widetilde{m}})$, is illustrated in the example in Figure \ref{fig:error-c} (for the case when $\gamma > 1/2$).
  \begin{figure}[H]
    \centering
        \centering
        \includegraphics[width=0.65\linewidth]{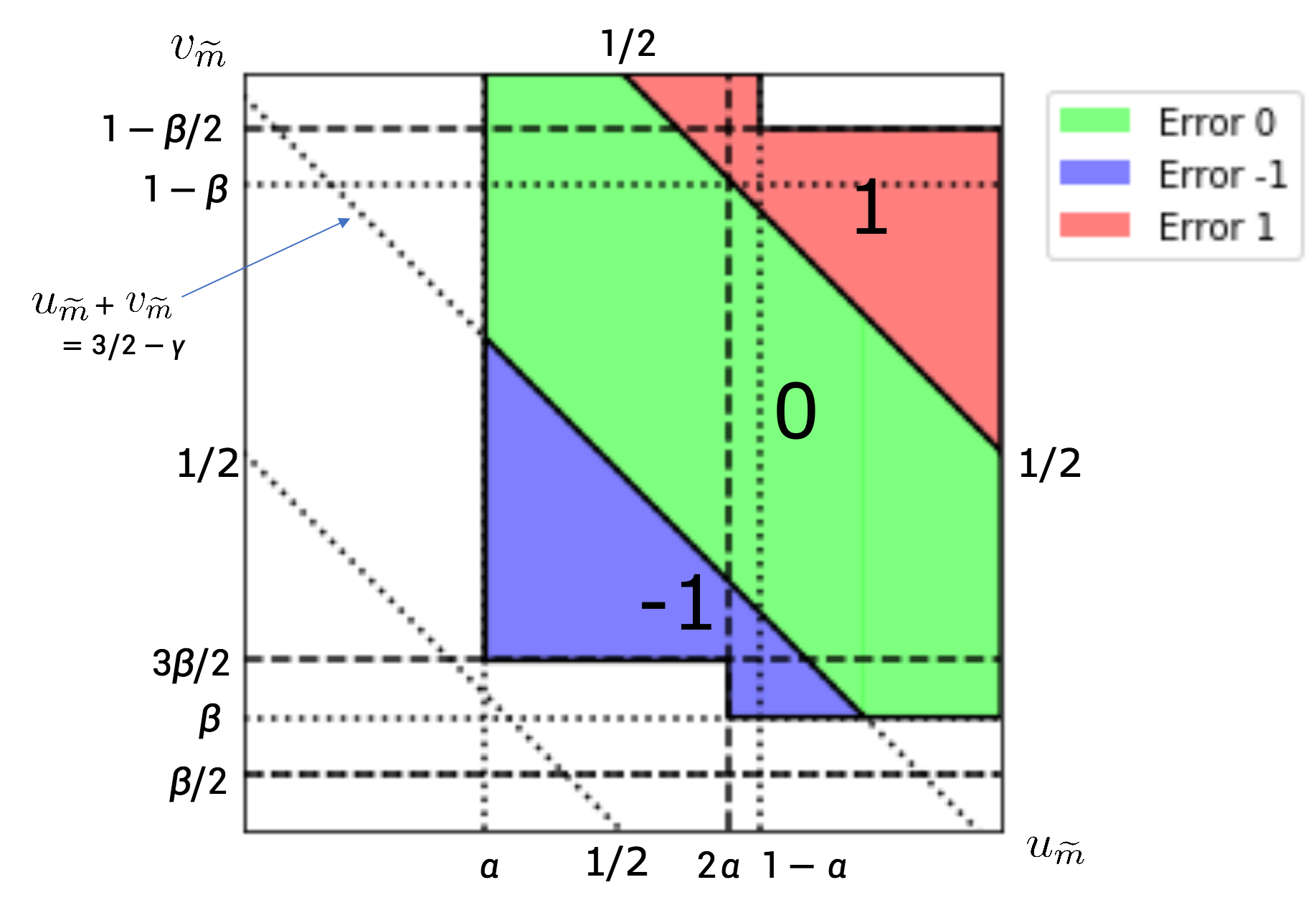}
    \caption{Distribution of $m-\widetilde{m}$  in terms of $(u_{\widetilde{m}}, v_{\widetilde{m}})$.}
    \label{fig:error-c}
\end{figure}
Our proof will show that the pairs $(u_{\widetilde{m}}, v_{\widetilde{m}})$ necessarily satisfy one of the three conditions in \eqref{eq:error-diff-c}.

The remainder of this section is devoted to the proof of Proposition \ref{prop:error-diff-c}. In Section \ref{subsec:proof-prop-c-sec6-lemmas} we prove some auxiliary lemmas and in Section \ref{subsec:proof-prop-c-sec6} we complete the proof of Proposition \ref{prop:error-diff-c}.


\subsection{Lemmas}
\label{subsec:proof-prop-c-sec6-lemmas}

\begin{lem}
\label{lem:inWct}
For any $\widetilde{m}\in\NN$, 
\begin{align}
\label{eq:inWct}
     \widetilde{m}\in \Wct &\Longleftrightarrow  u_{\widetilde{m}} > \alpha \text{ and } v_{\widetilde{m}} > \beta\\
     \notag
    &\qquad \text{ and } (u_{\widetilde{m}} < 1-\alpha \text{ or }v_{\widetilde{m}} < 1-\beta/2) \\
    \notag
   &\qquad \text{ and } (u_{\widetilde{m}} > 2\alpha \text{ or } v_{\widetilde{m}} > 3\beta/2).
\end{align}
\end{lem}
\begin{proof}
By the construction of Algorithm \ref{alg:2} and Proposition \ref{prop:alg-equivalent}, we have
\begin{align*}
    \widetilde{m} \in \Wat &\Longleftrightarrow u_{\widetilde{m}} <\alpha,\\
    \widetilde{m} \in \Wbt &\Longleftrightarrow u_{\widetilde{m}} >\alpha \text{ and } v_{\widetilde{m}} < \beta\\
    &\qquad \text{ or ($u_{\widetilde{m}} > 1-\alpha$ and $v_{\widetilde{m}} > 1-\beta/2$)} \\
   &\qquad \text{ or ($\alpha< u_{\widetilde{m}} < 2\alpha$ and $\beta< v_{\widetilde{m}}< 3\beta/2$)},\\
    \widetilde{m} \in \Wct &\Longleftrightarrow \widetilde{m} \notin (\Wat \cup \Wbt).
\end{align*}
Therefore,
\begin{align*}
 \widetilde{m}\in \Wct  &\Longleftrightarrow
 \widetilde{m}\not\in \Wa\text{ and } \widetilde{m}\not\in \Wbt\\    &\Longleftrightarrow u_{\widetilde{m}} > \alpha \text{ and }  (u_{\widetilde{m}} <\alpha \text{ or }   v_{\widetilde{m}} > \beta)\\
    &\qquad \text{ and $(u_{\widetilde{m}} < 1-\alpha$ or $v_{\widetilde{m}} < 1-\beta/2$)} \\
   &\qquad \text{ and ($u_{\widetilde{m}} < \alpha$ or $u_{\widetilde{m}} > 2\alpha$ or $v_{\widetilde{m}} < \beta$ or $v_{\widetilde{m}} > 3\beta/2$)}\\
   &\Longleftrightarrow \text{$u_{\widetilde{m}} > \alpha$ and $v_{\widetilde{m}} > \beta$}\\
    &\qquad \text{ and ($u_{\widetilde{m}} < 1-\alpha$ or $v_{\widetilde{m}} < 1-\beta/2$)} \\
    &\qquad \text{ and ($u_{\widetilde{m}} > 2\alpha$ or $v_{\widetilde{m}} > 3\beta/2$)}.
    \qedhere
\end{align*}
\end{proof}

\begin{lem}
\label{lem:c(n-1)<=ctn}
For any $n\in\NN$,
\begin{equation*}
c(n-1) \le \ct(n).
\end{equation*}
\end{lem}
\begin{proof}
We argue by contradiction. Suppose $\ct(n) \le c(n-1)-1$. Recall from Proposition \ref{prop:error-counting-c} that $\Ect(m) \in \{-1,0,1\}$ for any $m$. Then, by the monotonicity of the counting function $\Wct(m)$ and the fact that $\Wc(c(n-1)-1) = (n-1)-1 = n-2$, we have
\begin{align*}
n&=\Wct(\ct(n))\\
&\le \Wct(c(n-1)-1)\\
&=\Wc(c(n-1)-1)+\Ect(c(n-1)-1)\\
&=n-2+ \Ect(c(n-1)-1)\\
&\le n-1,
\end{align*}
which is a contradiction.
\end{proof}

\begin{lem}
\label{lem:ctn<=c(n+1)}
For any $n\in\NN$,
\begin{equation*}
\ct(n)\le c(n+1).
\end{equation*}
\end{lem}
\begin{proof}
We argue by contradiction. Suppose $\ct(n)\ge c(n+1)+1$. Then by the monotonicity of the counting function $\Wct(m)$ and the fact that $\Wc(c(n+1)) = n+1$, we have 
\begin{align*}
n-1&=\Wct(\ct(n)-1)\\
&\ge \Wct(c(n+1))\\
&=\Wc(c(n+1))+\Ect(c(n+1))\\
&= n+1 + \Ect(c(n+1))\\
&\ge n,
\end{align*}
which is a contradiction.
\end{proof}

\begin{lem}
\label{lem:cn-1<=ctn}
For any $n\in\NN$,
\begin{equation*}
c(n)-1\le \ct(n).
\end{equation*}
\end{lem}
\begin{proof}

We argue by contradiction. Suppose $ \ct(n) \le c(n)-2$. We consider, separately, the two cases for $\gamma > 1/2$ and $\gamma < 1/2$. 

Suppose first that $\gamma > 1/2$. Then $1<1/\gamma <2$. Hence, by Lemma \ref{lem:general-lemmas}(iii), the gap between any consecutive elements in $\Wc$ is equal to either 1 or 2. Therefore, $c(n)-2 \leq c(n-1)$. By Lemma \ref{lem:c(n-1)<=ctn}, we have $c(n-1) \leq \ct(n)$, so $\ct(n) \leq c(n)-2$ implies $c(n-1) = c(n)-2 =\ct(n)$. Then,
\begin{align*}
n&=\Wct(\ct(n))\\
&= \Wct(c(n-1))\\
&=\Wc(c(n-1))+\Ect(c(n-1))\\
&=n-1+ \Ect(c(n-1)).
\end{align*}
To obtain a contradiction, it now suffices to prove $\Ect(c(n-1))$ cannot be equal to 1. Denote $m = c(n-1) = \ct(n)$, and suppose $\Ect(m) = 1$. Then by Proposition \ref{prop:error-counting-c}, we have
\begin{equation}
    \label{eqproof:Ect=1equiv-m}
    3/2 < u_m+v_m <2 \text{ and } (u_m < 1-\alpha \text{ or } v_m < 1-\beta/2).
\end{equation}
By Lemma \ref{lem:general-lemmas}(iii),  $c(n-1) = c(n)-2$ implies $w_m < \fp{1/\gamma}\gamma = (1/\gamma-1)\gamma = 1-\gamma$. Then by \eqref{eq:w-further} and the first part of     \eqref{eqproof:Ect=1equiv-m},
\begin{align*}
    w_m < 1-\gamma &\Longleftrightarrow 1-\fp{u_m+v_m-1/2} < 1-\gamma\\
    &\Longleftrightarrow 1-(u_m+v_m-3/2) < 1-\gamma\\
    &\Longleftrightarrow u_m+v_m > 3/2+\gamma > 2,
\end{align*}
which is not possible since $u_m, v_m \in (0,1)$. This proves the result for the case $\gamma > 1/2$.

Now suppose $\gamma < 1/2$. By \eqref{eq:ab-less-half}, we have 
\begin{equation}
\label{eqproof:1/4<c<1/2}
    1/4 < \gamma < 1/2.
\end{equation}
Consider
\begin{align*}
n&=\Wct(\ct(n))\\
&\leq \Wct(c(n)-2)\\
&=\Wc(c(n)-2)+\Ect(c(n)-2)\\
&=n-1+ \Ect(c(n)-2).
\end{align*}
Denote $m' = c(n)-2$. To obtain a contradiction, it suffices to prove that $\Ect(m')$ can never attain the value 1. Suppose $\Ect(m') = 1$. Then by the same reasoning as for \eqref{eqproof:Ect=1equiv-m}, we have 
\begin{equation}
    \label{eqproof:Ect=1equiv-m'}
    3/2 < u_{m'}+v_{m'} <2 \text{ and } (u_{m'} < 1-\alpha \text{ or } v_{m'} < 1-\beta/2).
\end{equation}
By Lemma \ref{lem:general-lemmas}(i), the condition $m'+2 = c(n) \in \Wc$ is equivalent to $w_{m'+2} < \gamma$, which is further equivalent to $1-\fp{u_{m'+2}+v_{m'+2}-1/2} < \gamma$ by \eqref{eq:w-further}. We have
\begin{align*}
    1-\fp{u_{m'+2}+v_{m'+2}-1/2} &= 1-\fp{\fp{u_{m'}+2\alpha} + \fp{v_{m'}+2\beta} -1/2}\\
    &= 1-\fp{u_{m'}+2\alpha - \fl{u_{m'}+2\alpha} + v_{m'}+2\beta - \fl{v_{m'}+2\beta} -1/2}\\
    &= 1-\fp{u_{m'} +  v_{m'} +2(1-\gamma)  -1/2}\\
    &= 1-\fp{u_{m'}+  v_{m'}  -2\gamma   -1/2}.
\end{align*}
Thus, 
\begin{align*}
    w_{m'+2} < \gamma \Longleftrightarrow  1-\fp{u_{m'} +  v_{m'}-2\gamma  -1/2} < \gamma.
\end{align*}
By \eqref{eqproof:1/4<c<1/2} and     \eqref{eqproof:Ect=1equiv-m'}, we have $3/2 < u_{m'} + v_{m'} < 2$ and $1/4<\gamma < 1/2$. It implies $0 < u_{m'} +  v_{m'}-2\gamma  -1/2 < 1$. Hence,
\begin{align*}
   w_{m'+2} < \gamma &\Longleftrightarrow 1-(u_{m'} +  v_{m'}-2\gamma  -1/2) < \gamma\\
   &\Longleftrightarrow u_{m'} +  v_{m'} > 3/2 + \gamma .
\end{align*}
We now prove that the latter inequality contradicts the condition $(u_{m'} < 1-\alpha \text{ or } v_{m'} < 1-\beta/2)$ in \eqref{eqproof:Ect=1equiv-m'}. When $u_{m'} < 1-\alpha$, because $\beta < 1/2$ (see \eqref{eq:ab-less-half}), we must have
\begin{align*}
    v_{m'}  &> 3/2 + \gamma - (1-\alpha)\\
    &= 3/2-\beta\\
    &> 1,
\end{align*}
which is not possible. When $v_{m'} < 1-\beta/2$, because $\alpha < \gamma$ and $\alpha+\beta+\gamma = 1$ (see \eqref{eq:abc-condition}) imply that  $\gamma + \beta/2 > 1/2$, we must have
\begin{align*}
    u_{m'}  &> 3/2 + \gamma - (1-\beta/2)\\
    &= 1/2 + \gamma + \beta/2\\
    &> 1,
\end{align*}
which is again not possible. This gives the desired contradiction and thus completes the proof.
\end{proof}

\begin{lem}
\label{lem:Ect-implies-nIn}
For any $m\in\NN$,
\begin{align}
    \label{lem:EctM1-implies-nInWc}
   &\text{ if $E_\gamma(m)= -1$, then $m \in \Wc$;}\\
   \label{lem:Ect1-implies-nInWct}
   &\text{ if $E_\gamma(m)= 1$, then $m \in \Wct$.}
\end{align}
\end{lem}

\begin{proof}
To prove \eqref{lem:EctM1-implies-nInWc}, suppose we have $\Ect(m) = -1$. By Proposition \ref{prop:error-counting-c}, this implies
\begin{equation}
\label{eqproof:Ect-M1}
      0 < u_m+v_m < 1/2 \text{ and } ( u_m > \alpha \text{ or } v_m > \beta/2).
\end{equation}
By Lemma \ref{lem:general-lemmas}(i) and \eqref{eq:w-further}, $m \in W_\gamma$ is equivalent to $w_m<\gamma$, where $w_m = 1-\{u_m+v_m-1/2\}$ and $\gamma = 1-\alpha-\beta$. Under the condition \eqref{eqproof:Ect-M1}, we thus have
\begin{align*}
    m \in \Wc &\Longleftrightarrow1-\{u_m+v_m-1/2\} < 1-\alpha-\beta\\
     \notag
     &\Longleftrightarrow  1-(u_m+v_m+1/2)
     <1-\alpha-\beta\\
     \notag
     &\Longleftrightarrow u_m+v_m > \alpha+\beta - 1/2.
\end{align*}
Therefore, it suffices to show
\begin{align}
\label{eqproof:nInWc}
 u_m+v_m > \alpha+\beta - 1/2.
\end{align}
We consider the cases $u_m > \alpha$ and $v_m > \beta/2$ in \eqref{eqproof:Ect-M1} separately. When $u_m > \alpha$, since $\beta < 1/2$ by \eqref{eq:ab-less-half}, we have
\begin{align*}
    u_m+v_m > u_m > \alpha > \alpha + \beta-1/2,
\end{align*}
which proves \eqref{eqproof:nInWc}
 in this case.
When $v_m > \beta/2$, since $\alpha+\beta/2 <1/2$ by \eqref{eqproof:a+b/2<1/2}, we have
\begin{align*}
   u_m+v_m > v_m > \beta/2 > \alpha + \beta-1/2,
\end{align*}
which proves \eqref{eqproof:nInWc} in the case for $v_m > \beta/2$.

Now we prove  \eqref{lem:Ect1-implies-nInWct}. Suppose $\Ect(m) = 1$ for some $m$. By Proposition \ref{prop:error-counting-c}, this implies
\begin{equation}
    \label{Egamma=1_imply}
    3/2 < u_m+v_m <2 \text{ and } ( u_m < 1-\alpha \text{ or } v_m < 1-\beta/2).
\end{equation}
Recall from Lemma \ref{lem:inWct} that \begin{align*}
     m\in \Wct &\Longleftrightarrow \text{$u_m > \alpha$ and $v_m > \beta$}\\
     \notag
    &\qquad\text{ and ($u_m < 1-\alpha$ or $v_m < 1-\beta/2$)}\\
    \notag
   &\qquad\text{ and ($u_m > 2\alpha$ or $v_m > 3\beta/2$)}
\end{align*}
Clearly $\Ect(m) = 1 $ implies $ u_m > \alpha, v_m > \beta \text{ and } ( u_m < 1-\alpha \text{ or } v_m < 1-\beta/2)$, since $\alpha,\beta < 1/2$ and $u_m, v_m < 1$. It suffices to prove that $\Ect(m) = 1$ also implies $(u_m > 2\alpha \text{ or } v_m > 3\beta/2)$. We argue by contradiction. Suppose $u_m < 2\alpha$ and $v_m < 3\beta/2$. Then, by \eqref{eq:ab-less-half}, we have
\begin{align*}
    u_m+v_m &< 2\alpha + \frac{3\beta}{2}\\\notag
    &= \frac{\alpha}{2} + \frac{3}{2} (1-\gamma)\\\notag
    &< \frac{1/2}{2} +\frac{3}{2} (1-\frac{1}{4})\\\notag
    &= \frac{11}{8}.
\end{align*}
which contradicts $u_m+v_m > 3/2$ given by \eqref{Egamma=1_imply}. This completes the proof.

\end{proof}
\begin{lem}
\label{lem:ctn<=cn+1}
For any $n\in\NN$,
\begin{equation}
\ct(n) \leq c(n)+1.
\end{equation}
\end{lem}
\begin{proof}
We argue by contradiction. Suppose $\ct(n) \geq c(n)+2$. Then 
\begin{align}
\label{eqproof:ctn-righthalf}
    n-1 &= \Wct(\ct(n)-1) \\
    \notag
    &\geq \Wct(c(n)+1)\\
    \notag
    &= \Wc(c(n)+1) + \Ect(c(n)+1)\\
    \notag
    &\geq n + \Ect(c(n)+1).
\end{align}
Note that \eqref{eqproof:ctn-righthalf} implies $\Ect(c(n)+1) = -1$ and $\Wc(c(n)+1) = n$. By Lemma \ref{lem:Ect-implies-nIn}, $\Ect(c(n)+1) = -1$ implies $c(n)+1 \in \Wc$ and therefore $\Wc(c(n)+1) = n+1$, which contradicts $\Wc(c(n)+1) = n$.
\end{proof}

\begin{lem}
\label{lem:c-error-characterize-original}
Given $n\in\mathbb{N}$, set $m = c(n), \widetilde{m}=\ct(n)$. Then $m-\widetilde{m} = \{-1, 0, 1\}$ and 
\begin{equation*}
m-\widetilde{m}=
\begin{cases}
-1&\text{if $\widetilde{m} \in \Wct$ and ($\Ect(\widetilde{m}) = -1$ or
($\Ect(\widetilde{m}) = 0$ and $\widetilde{m} \notin \Wc$))};
\\
0&\text{if
$\widetilde{m} \in \Wct$ and $\Ect(\widetilde{m}) = 0$  and $\widetilde{m} \in \Wc$};
\\
1&\text{if
$\widetilde{m} \in \Wct$ and $\Ect(\widetilde{m}) = 1$}.
\end{cases}
\end{equation*}
\end{lem}
\begin{proof}
$m-\widetilde{m} = \{-1, 0, 1\}$ follows from Lemma \ref{lem:cn-1<=ctn} and Lemma \ref{lem:ctn<=cn+1}. Moreover, by Proposition \ref{prop:error-counting-c}, it is easy to see that the conditions on the right exhaust all possibilities. Since $m = c(n), \widetilde{m}=\ct(n)$, we have $\Wct(\widetilde{m})=n = \Wc(m)$. We break the discussion into three cases according to the value of $m$.
 
\textit{Case I: $m= \widetilde{m}$}. In this case, $\Wct(\widetilde{m}) = n = \Wc(\widetilde{m})$, so $\Ect(\widetilde{m}) = \Wct(\widetilde{m}) - \Wc(\widetilde{m}) = 0$.

\textit{Case II: $m= \widetilde{m}+1$}. In this case,
\begin{align*}
    \Ect(\widetilde{m}) &= \Wct(\widetilde{m}) - \Wc(\widetilde{m})\\
    &= \Wct(\widetilde{m}) - \Wc(m-1)\\
    &=n - (n-1)\\
    &= 1.
\end{align*}

\textit{Case III: $m= \widetilde{m}-1$}. In this case,
\begin{align*}
    \Ect(\widetilde{m}) &= \Wct(\widetilde{m}) - \Wc(\widetilde{m})\\
    &= \Wct(\widetilde{m}) - \Wc(m+1)\\
    &= 0 \text{ or } -1.
\end{align*}

Note that by Proposition \ref{prop:error-counting-c}, $\Ect(\widetilde{m}) \in \{-1, 0, 1\}$. Thus, we have that $\Ect(\widetilde{m})=1$ implies $m=\widetilde{m}+1$ and $\Ect(\widetilde{m})=-1$ implies $m=\widetilde{m}-1$. When $\Ect(\widetilde{m})=0$, there are two possible situations: $m=\widetilde{m}-1$ or $m=\widetilde{m}$. In fact, given $\Ect(\widetilde{m})=0$, $\widetilde{m} \in \Wc$ if and only if $m=\widetilde{m}$. The ``if" part is clear since $m \in \Wc$. We prove the``only if" part by contradiction. Suppose $\widetilde{m} \in \Wc$ and $m=\widetilde{m}-1$ hold simultaneously, then $\Wc(\widetilde{m}) = n+1$ and thus $\Ect(\widetilde{m}) = \Wct(\widetilde{m}) - \Wc(\widetilde{m}) = n - (n+1) = -1$, contradicting $\Ect(\widetilde{m})=0$. Hence, the condition to differentiate the two situations is whether $\widetilde{m} \notin \Wc$ or $\widetilde{m} \in \Wc$. Combining all above gives the error characterization conditions in this lemma.
\end{proof}

\subsection{Proof of Proposition \ref{prop:error-diff-c}
}
\label{subsec:proof-prop-c-sec6}
By Lemma \ref{lem:c-error-characterize-original}, we have $m-\widetilde{m} \in \fp{-1,0,1}$, so it remains to show that the conditions for the three cases for the errors $m-\widetilde{m}$ in Lemma \ref{lem:c-error-characterize-original} are equivalent to the conditions \eqref{eq:error-diff-c} in Proposition \ref{prop:error-diff-c}.

\textit{Case I: $m-\widetilde{m} = 1$.} By Lemma \ref{lem:c-error-characterize-original}, $m-\widetilde{m} = 1$ is equivalent to 
\begin{equation}
    \label{cn-ctn=1-equiv}
    \widetilde{m}\in \Wct \text{ and } \Ect(\widetilde{m})=1.
\end{equation}
In fact, by Lemma \ref{lem:Ect-implies-nIn}, $\Ect(\widetilde{m})=1$ implies $\widetilde{m}\in \Wct$, so     \eqref{cn-ctn=1-equiv} is equivalent to $\Ect(\widetilde{m})=1$. Thus, by Proposition \ref{prop:error-counting-c}, 
\begin{align*}
    m-\widetilde{m} = 1 &\Longleftrightarrow \Ect(\widetilde{m})=1\\
    &\Longleftrightarrow 3/2 < u_{\widetilde{m}}+v_{\widetilde{m}} <2 \text{ and } ( u_{\widetilde{m}} < 1-\alpha \text{ or } v_{\widetilde{m}} < 1-\beta/2).
\end{align*}

\textit{Case II: $m-\widetilde{m} = 0$.} By Lemma \ref{lem:c-error-characterize-original},  $m-\widetilde{m} = 0$ is equivalent to
\begin{equation}
    \label{pfProp6.1-eqN1}
     \widetilde{m}\in \Wct \text{ and } \Ect(\widetilde{m})=0 \text{ and  } \widetilde{m} \in\Wc.
\end{equation}
 Note that when $\widetilde{m}\in \Wct$, by Lemma \ref{lem:inWct} and Proposition \ref{prop:error-counting-c}, we have
 \begin{equation}
 \label{eqproof:assume-mInWct-Ect=0}
    \Ect(\widetilde{m})=0 \Longleftrightarrow 1/2 < u_{\widetilde{m}}+v_{\widetilde{m}} < 3/2.
 \end{equation}
 Moreover, \eqref{eq:w-further} and Lemma \ref{lem:general-lemmas}(i) give that 
 \begin{equation}
    \label{pfProp6.1-eqN2}
     \widetilde{m} \in\Wc \Longleftrightarrow w_{\widetilde{m}}<\gamma \Longleftrightarrow \fp{u_{\widetilde{m}}+v_{\widetilde{m}}-1/2} > 1-\gamma.
\end{equation}
 Hence by \eqref{pfProp6.1-eqN1}, Lemma \ref{lem:inWct},  \eqref{eqproof:assume-mInWct-Ect=0}, and  \eqref{pfProp6.1-eqN2},
\begin{align*}
    m-\widetilde{m} = 0 &\Longleftrightarrow \widetilde{m}\in \Wct \text{ and } \Ect(\widetilde{m})=0 \text{ and }\widetilde{m} \in\Wc.\\
    &\Longleftrightarrow 1/2 < u_{\widetilde{m}}+v_{\widetilde{m}} < 3/2, u_{\widetilde{m}} > \alpha, v_{\widetilde{m}} > \beta, \text{ ($u_{\widetilde{m}} < 1-\alpha$ or $v_{\widetilde{m}} < 1-\beta/2$)},\\
    &\qquad\text{ ($u_{\widetilde{m}} > 2\alpha$ or $v_{\widetilde{m}} > 3\beta/2$), $\fp{u_{\widetilde{m}}+v_{\widetilde{m}}-1/2} > 1- \gamma$}.
\end{align*}
When we assume $1/2 < u_{\widetilde{m}}+v_{\widetilde{m}} < 3/2$, we have
\[\fp{u_{\widetilde{m}}+v_{\widetilde{m}}-1/2} > 1- \gamma \Longleftrightarrow u_{\widetilde{m}}+v_{\widetilde{m}} > 3/2-\gamma.
\]
Note that $(u_{\widetilde{m}} > 2\alpha \text{ or } v_{\widetilde{m}} > 3\beta/2)$ is already implied by $ u_{\widetilde{m}}+v_{\widetilde{m}} > 3/2-\gamma$.  Suppose $u_{\widetilde{m}} < 2\alpha \text{ and } v_{\widetilde{m}} < 3\beta/2 $, then $2\alpha + 3\beta/2 > u_{\widetilde{m}}+v_{\widetilde{m}} > 3/2-\gamma$. It implies that  $\alpha > \gamma$, which contradicts condition \eqref{eq:abc-condition}. Hence, condition $(u_{\widetilde{m}} > 2\alpha \text{ or } v_{\widetilde{m}} > 3\beta/2)$ is redundant. Furthermore, $ (u_{\widetilde{m}} < 1-\alpha \text{ or }v_{\widetilde{m}} < 1-\beta/2)$ is also redundant, because if we assume $ (u_{\widetilde{m}} > 1-\alpha \text{ and }v_{\widetilde{m}} > 1-\beta/2)$, then $u_{\widetilde{m}}+v_{\widetilde{m}} > 1-\alpha + 1-\beta/2 > 3/2$ by \eqref{eqproof:a+b/2<1/2}, which contradicts $1/2<u_{\widetilde{m}}+v_{\widetilde{m}} < 3/2$.
Therefore, we get
\begin{align*}
    m-\widetilde{m} = 0 &\Longleftrightarrow 1/2 < u_{\widetilde{m}}+v_{\widetilde{m}} < 3/2, u_{\widetilde{m}} > \alpha, v_{\widetilde{m}} > \beta, u_{\widetilde{m}}+v_{\widetilde{m}} > 3/2-\gamma\\
    &\Longleftrightarrow  \text{$u_{\widetilde{m}}>\alpha$ and  $v_{\widetilde{m}}>\beta$ and $3/2-\gamma < u_{\widetilde{m}}+v_{\widetilde{m}} < 3/2$.}
\end{align*}

\textit{Case III: $m-\widetilde{m} = -1$}. By Lemma \ref{lem:c-error-characterize-original},
\begin{align}
    \label{eqproof:cn-ctn=-1Condition}
    m-\widetilde{m} = -1 &\Longleftrightarrow \widetilde{m} \in \Wct \text{ and } (E_\gamma(\widetilde{m}) = -1 \text{ or } (E_\gamma(\widetilde{m}) = 0 \text{ and } \widetilde{m} \notin W_\gamma)).
 \end{align}
 By Proposition \ref{prop:error-counting-c}, \eqref{eqproof:assume-mInWct-Ect=0}, and \eqref{eq:w-further},
we have
\begin{align*}
    E_\gamma(\widetilde{m}) = -1 \text{ or } &(E_\gamma(\widetilde{m}) = 0 \text{ and } \widetilde{m} \notin W_\gamma)\\
    &\Longleftrightarrow \text{ ($0 < u_{\widetilde{m}}+v_{\widetilde{m}} < 1/2$ and ($u_{\widetilde{m}} > \alpha$ or $v_{\widetilde{m}} > \beta/2$))}\\
    &\qquad\text{ or ($1/2<u_{\widetilde{m}}+v_{\widetilde{m}}<3/2$ and $\fp{u_{\widetilde{m}}+v_{\widetilde{m}}-1/2}<\alpha+\beta$)}\\
    &\Longleftrightarrow  \text{ ($0 < u_{\widetilde{m}}+v_{\widetilde{m}} < 1/2$ and ($u_{\widetilde{m}} > \alpha$ or $v_{\widetilde{m}} > \beta/2$))}\\
    &\qquad\text{ or ($1/2<u_{\widetilde{m}}+v_{\widetilde{m}}<3/2$ and $u_{\widetilde{m}}+v_{\widetilde{m}}<3/2-\gamma$)}.
\end{align*}
 Consider the equivalent condition of $\widetilde{m}\in\Wct$ in \eqref{eq:inWct}, observe that $(u_{\widetilde{m}}>1-\alpha \text{ and } v_{\widetilde{m}} > 1-\beta/2)$ implies $u_{\widetilde{m}}+v_{\widetilde{m}} > 3/2$, which is impossible by the previous formula, so we can eliminate the second conjunct in  the equivalent condition for $\widetilde{m} \in \Wct$. Then we have, by \eqref{eq:inWct} and  \eqref{eqproof:cn-ctn=-1Condition},
 \begin{align*}
    m-\widetilde{m} = -1  &\Longleftrightarrow  \text{$u_{\widetilde{m}} > \alpha$ and $v_{\widetilde{m}} > \beta$ and ($u_{\widetilde{m}} > 2\alpha$ or $v_{\widetilde{m}} > 3\beta/2$)}\\
   &\qquad\text{ and (($0 < u_{\widetilde{m}}+v_{\widetilde{m}} < 1/2$ and ($u_{\widetilde{m}} > \alpha$ or $v_{\widetilde{m}} > \beta/2$))}\\
    &\qquad\text{ or ($1/2<u_{\widetilde{m}}+v_{\widetilde{m}}<3/2$ and $u_{\widetilde{m}}+v_{\widetilde{m}}<3/2-\gamma$))}\\
     &\Longleftrightarrow \text{$u_{\widetilde{m}} > \alpha$ and $v_{\widetilde{m}} > \beta$ and ($u_{\widetilde{m}} > 2\alpha$ or $v_{\widetilde{m}} > 3\beta/2$)}\\
   &\qquad\text{ and $u_{\widetilde{m}}+v_{\widetilde{m}} < 3/2-\gamma$}.
\end{align*}
which is the condition in \eqref{eq:error-diff-c}.
This completes the proof of Proposition \ref{prop:error-diff-c}. \qed


\section{Optimality}
\label{sec:optimality}
In this section, we show that Theorem \ref{thm:main} is best-possible in several respects. Using a result of Graham, we first show in Theorem \ref{thm:optimality:TwoExactOneAlmost} below that, under a mild linear independence condition on the densities $\alpha, \beta, \gamma$, a partition of $\NN$ into three or more sequences can involve at most one (exact) Webster sequence.
\begin{lem}[Graham \cite{Graham1973}]
\label{lem:disjointness}
  Suppose $\Wa,\Wb$ are disjoint. Then either
    \begin{itemize}
    \item [(i)] $\alpha/\beta$ is rational; or
    \item [(ii)] there exist positive integers $r,s$ such that $r\alpha+s\beta=1$ and $r\equiv s \mod 2$.
    \end{itemize}
\end{lem}
\begin{proof}
This is Fact 3 in \cite{Graham1973}, in a slightly different notation.
\end{proof}

\begin{thm}
\label{thm:optimality:TwoExactOneAlmost}

Suppose $\alpha,\beta$ satisfy
\begin{equation}
\label{cond:independent}
\text{$1,\alpha,\beta$ are linearly independent over $\QQ$}.
\end{equation}
Then there does not exist a  partition of $\NN$ into three or more sequences involving the two exact Webster sequences $\Wa$ and $\Wb$. 

\end{thm}
\begin{proof}
	We argue by contradiction. Suppose $\alpha, \beta$ satisfy \eqref{cond:independent}
 and $\Wa, \Wb$ are two exact Webster sequences in a three-part partition of $\NN$. Then $\Wa, \Wb$ must be disjoint. Moreover, \eqref{cond:independent} implies that $\alpha/\beta \notin \QQ$ and that there do not exist $r,s\in\NN$ with $r\alpha + s\beta = 1$. But this contradicts Lemma \ref{lem:disjointness}.
\end{proof}

Next, we show that the error bounds $-1 \leq \bt(n) - b(n) \leq 1$ in Theorem \ref{thm:main} are best-possible.

\begin{thm}[Optimality of error bounds]
\label{thm:optimality:one-sided-impossible}

Suppose $\alpha,\beta,\gamma $ satisfy \eqref{eq:abc-condition}, \eqref{cond:independent} and in addition,
\begin{equation}
\label{eq:one-sided-impossible-conditions}
 \alpha > 1/3 \text{ and } \gamma < 1/2.\\
\end{equation}
 Then there does not exist a partition of $\NN$ into sequences $\Wa, \Wbt, \Wct$ such that
 \begin{equation}
 \label{optimaliy-errorbound-c}
     \text{$\ct(n)-c(n)\in \{-1,0,1\}$ for all $n\in\NN$}
 \end{equation}
 and either
  \begin{equation}
      \label{optimaliy-errorbound-b-M1}
      \text{ $\bt(n)-b(n) \in \{-1,0\}$ for all $n\in\NN$}
 \end{equation}
 or
  \begin{equation}
    \label{optimaliy-errorbound-b-P1}
    \text{ $\bt(n)-b(n) \in \{0,1\}$ for all $n\in\NN$.}
 \end{equation}
\end{thm}

\begin{proof} We argue by contradiction. We only consider the case \eqref{optimaliy-errorbound-b-P1}, as the analysis for the case \eqref{optimaliy-errorbound-b-M1} is similar. Suppose $\Wa, \Wbt, \Wct$ is a partition of $\NN$ such that \eqref{optimaliy-errorbound-c} and \eqref{optimaliy-errorbound-b-P1} hold. We will show that there exists an $m$ such that 
\begin{equation}
\label{eqproof:1/3<a<1/2Condition}
    m\in \Wa \cap \Wb, m+1 \in \Wc, m+2 \in\Wa.
\end{equation}
Assuming \eqref{eqproof:1/3<a<1/2Condition}, we have $m = b(n),  m+1= c(k)$ for some $n$ and $k$. Since $m ,m+2 \in \Wa$, and $\Wa, \Wbt, \Wct$ partition $\NN$, it follows that $m ,m+2 \notin \Wbt$ and $m ,m+2 \notin \Wct$. 
By our assumptions \eqref{optimaliy-errorbound-c} and \eqref{optimaliy-errorbound-b-P1}, we then have $\bt(n) \in \{m ,m+1\}$ and $\ct(k) \in \{m ,m+1, m+2\}$. Since $m, m+2 \in \Wa$, this implies $\bt(n) = m+1 = \ct(k)$, contradicting the partition property. Therefore, it suffices to prove that there exists an $m$ satisfying \eqref{eqproof:1/3<a<1/2Condition}.

Lemma \ref{lem:general-lemmas}(i) gives that
\begin{equation*}
    m \in \Wb \Longleftrightarrow v_m < \beta.
\end{equation*}
We have $ m+1 \in\Wc \Longleftrightarrow w_m > 1-\gamma$. From \eqref{eq:w-further} we get
\begin{align*}
    m+1 \in\Wc &\Longleftrightarrow  1-\fp{u_m+v_m-1/2} > 1-\gamma\\
     &\Longleftrightarrow
    \text{$1/2 <u_m+v_m< \gamma + 1/2$}\\
     &\qquad\text{ or $3/2 < u_m+v_m < 3/2 + \gamma$.}
\end{align*}
By Lemma \ref{lem:general-lemmas}(i) and (iii),
\[
m, m+2\in\Wa \Longleftrightarrow \fp{1/\alpha}\alpha < u_m < \alpha.
\]
Note here we have used that $\alpha < \gamma < 1/2$, so the gap between any two successive elements in $\Wa$ is at least 2. Therefore \eqref{eqproof:1/3<a<1/2Condition} is equivalent to
\begin{equation}
\label{eqproof:1/3<a<1/2Condition-equiv}
    \begin{cases}
    &\fp{1/\alpha}\alpha < u_m < \alpha,\\
    &1/2 <u_m+v_m< \gamma + 1/2 \text{ or }3/2 < u_m+v_m < 3/2 + \gamma,\\
    &v_m <\beta.
\end{cases}
\end{equation}
By Corollary \ref{cor:unif-ab} and the linear independence of $1, \alpha, \beta$, for general $n\in\NN$, the pairs $(u_m, v_m)$ are uniformly distributed modulo 1. By choosing $u_m > \alpha - \epsilon, v_m > \beta-\epsilon$ for a suitable positive $\epsilon$, one can check that
the regions inside the unit square defined by
the three conditions in \eqref{eqproof:1/3<a<1/2Condition-equiv} have a nonempty intersection. Hence, there exists an $m$ satisfying \eqref{eqproof:1/3<a<1/2Condition}. This completes the proof.
\end{proof}

\section{Concluding Remarks}
\label{sec:concluding-remarks}
    In this section, we discuss some possible generalizations and extensions of our results.


\paragraph{Necessity of the conditions in Theorem \ref{thm:main}.}
 The conclusion of Theorem \ref{thm:main} does not necessarily hold for the sequences generated by Algorithm \ref{alg:1} if the condition $\beta < 1/2$ or $\alpha < \gamma$ in Theorem \ref{thm:main} is not satisfied. For example, a computer search shows that when  $(\alpha, \beta) = (\sqrt{2} / 10, 3\sqrt{3} / 10) = (0.14142\dots, 0.51961\dots)$, the sequences generated by Algorithm \ref{alg:1} satisfy $|\ct(n)-c(n)| = 2$ for $n = 692$. Similarly, when $(\alpha, \beta) = ( 3\sqrt{2} / 10, \sqrt{3} / 10) = (0.42426\dots, 0.17320\dots)$ (which does not satisfy the condition $\alpha < \gamma$), the sequences generated by Algorithm \ref{alg:1} satisfy $|\ct(n)-c(n)| = 2$ for $n = 4787$. However, it may still be possible that, with a different construction, Theorem \ref{thm:main} remains valid under more general conditions on $\alpha,\beta$ and $\gamma$.

\paragraph{Partitions into two exact Webster sequences and one 
almost Webster sequence.}
By Theorem \ref{thm:optimality:TwoExactOneAlmost}, in general it is not possible to partition $\NN$ into two Webster sequences and one almost Webster sequence with given densities $\alpha, \beta, \gamma$. Nevertheless, for certain special irrational triples $(\alpha, \beta, \gamma)$, we can obtain such a partition. In fact, we have the following theorem, which is analogous to Theorem 1 in \cite{beatty-paper} and can be proved using similar methods.

\begin{thm}
\label{thm:TwoEact-rs}
Suppose $\alpha, \beta, \gamma$ satisfy
\eqref{eq:abc-condition} and there exist positive integers $r,s$ such that 
\begin{equation}
\label{eq:constr1-condition}
r\alpha+s\beta=1,\quad r \equiv s \mod 2,\quad r,s\ge 2
\end{equation}
If we define $\Wct = \NN \setminus (\Wa\cup\Wb)$, then $\Wa, \Wb, \Wct$ partition $\NN$. Moreover, 
\[
\ct(n)-c(n)\in\{-1,0,1\} \quad (n\in\NN)
\]
and thus $\Wct$ is an almost Webster sequence.

\end{thm}

\paragraph{The case of finite sequences and rational densities.}
In real-life applications, we seek to partition a finite sequence $\{1,2,\dots, N\}$ into sequences of length $d_i$ with $\sum_{i=1}^k d_i = N$ as evenly as possible. Then the associated densities $\alpha_i = d_i/N$ are rational, so our results do not directly apply. An approximation argument shows that when some of $\alpha, \beta, \gamma$ are rational but satisfy the other conditions of Theorem \ref{thm:main}, one can obtain partitions into perturbed Webster sequences of densities $\alpha, \beta, \gamma$ whose perturbation errors are at most 1 larger than those in Theorem \ref{thm:main}, i.e. at most 2. One can ask if one can obtain perturbation errors at most 1 in the case of rational densities. In fact, a computer search showed that when $(\alpha, \beta, \gamma)  = (7/20, 3/20, 1/2)$, the sequences constructed\footnote{When $\alpha$ and $\beta$ are rational, it is possible that $u_m = \alpha$ or $v_m = \beta$, so Algorithm \ref{alg:1} needs to be extended to cover these cases. We checked all possible extensions in our computer search, and in all cases, errors $\pm2$ occurred for some $n < 50$ for the example $(\alpha, \beta, \gamma)  = (7/20, 3/20, 1/2)$.} by Algorithm \ref{alg:1} satisfy $|\ct(n)-c(n)| \ge 2$ for some $n < 50$, so the conclusion \eqref{eq:mainThm-bounds-c} of Theorem \ref{thm:main} does not hold, at least for the sequences constructed by Algorithm \ref{alg:1}.

\paragraph{Partitions into $k$ sequences.} A natural extension of our results would be partitions of $\NN$ into more than 3 sequences. However, by Theorem \ref{thm:optimality:TwoExactOneAlmost}, in general at most one of them can be a Webster sequence. In fact, as pointed out by the referee, in general, perturbations of size at least $\cl{k/2}$ are necessary to obtain a partition of $\NN$ into $k$ sequences. This is obvious in the case when all densities $\alpha_i$ are equal, i.e., $\alpha_i = 1/k$ for $i=1, \dots, k$, since in this case, the Webster sequence $W_{\alpha_i} = W_{1/k} =  \{\cl{k/2}, \cl{3k/2}, \cl{5k/2}, \dots \}$ has gaps of size $k$, so in order to get a partition, perturbations of $\cl{k/2}$ are necessary.

\section{Acknowledgments}
I am grateful to Professor A.J. Hildebrand for providing motivation, supervision, and suggestions on this research project in the past two years. I would also like to express my gratitude to the opportunity provided by the \emph{Illinois Geometry Lab} at the University of Illinois, where this research originated in Spring 2018. I thank the referee for many helpful comments and suggestions that helped improve the exposition and shorten some of the arguments.
\bibliographystyle{plain}


\Addresses

\end{document}